\documentclass[a4paper, 11pt]{article}
\usepackage[a4paper,margin=2cm]{geometry}

\usepackage{amsmath, amsthm, amssymb, amsfonts}
\usepackage{enumerate}
\usepackage[hidelinks]{hyperref}
\usepackage{mathrsfs}
\usepackage{float}
\usepackage{caption}
\usepackage{subcaption}
\usepackage{authblk}

\usepackage[utf8]{inputenc}
\usepackage[T1]{fontenc}
\usepackage[french, english]{babel}

\numberwithin{equation}{section}

\usepackage{tikz}
\usetikzlibrary{arrows.meta,  decorations.markings,  calc,  positioning,  trees}
\usepackage{tikz-cd}

\newtheorem{theorem}{Theorem}[section]
\newtheorem{lemma}[theorem]{Lemma}
\newtheorem{proposition}[theorem]{Proposition}
\newtheorem{corollary}[theorem]{Corollary}

\theoremstyle{definition}
\newtheorem{definition}[theorem]{Definition}

\theoremstyle{remark}
\newtheorem{remark}[theorem]{Remark}
\newtheorem{example}[theorem]{Example}

\title{The Isoperimetric Problem in Regular Trees}  
\author{ Marc Troyanov\footnote{\texttt{marc.troyanov@epfl.ch},  \ {\'E}cole Polytechnique F\'ederale de Lausanne, Switzerland.}}
\date{April 20, 2026}

\begin{document}

\maketitle

\begin{abstract}
We investigate the inner vertex-isoperimetric problem on the $d$-regular tree $T_d$.
For each integer $k\ge 1$, we determine the exact value of the inner vertex-isoperimetric profile
\[
   I_d(k) = \min\{ |\partial D| \mid D\subset T_d \text{ finite and connected},\ |D|=k \}.
\]
We introduce a boundary invariant, called the \emph{boundary branching excess} $\tau(D)$, and show that it provides a simple criterion for optimality. A  domain $D\subset T_d$ is isoperimetrically optimal if and only if $\tau(D)\le d-2$. Finally, we show that every domain in $T_d$ admits a canonical decomposition as an iterated gluing of \emph{full domains}, namely domains whose entire boundary consists of leaves. This yields a complete description of all inner vertex-isoperimetric minimizers in  $T_d$.

\medskip


\smallskip

\end{abstract}

{
\setcounter{tocdepth}{2}
\small
\tableofcontents
}

\section{Introduction}
The present paper addresses the inner vertex-isoperimetric problem on the regular tree $T_d$ and provides an exact solution together with a complete structural description of optimal domains.
This investigation grew out of material developed during the PhD thesis of Bruno Luiz Santos Correia at EPFL \cite{CorreiaThesis2025}. The thesis contains initial results and observations on the inner isoperimetric problem on the regular tree, including a computation of the isoperimetric profile. In the present paper, we develop these ideas into a complete and systematically organized treatment.

\subsection{Statement of the problem}

The \emph{isoperimetric problem} in a locally finite graph $X$ addresses the following general question: among all domains $D \subset X$ (that is, finite
connected sets of vertices) of fixed cardinality, which ones have the smallest possible boundary? This problem naturally splits into two parts:
\begin{enumerate}[(i)]
  \item What is the smallest possible size of the boundary $\partial D$ of a
        domain $D \subset X$ of given cardinality?
  \item Describe all domains that are \emph{optimal}, that is, those which attain this minimum.
\end{enumerate}

Question~(i) corresponds to determining the best possible isoperimetric inequality in $X$, encoded by the \emph{isoperimetric profile}
$$
   I_X(k) = \min \left\{\, |\partial D| \mid D \subset X \text{ finite and connected},\ |D| = k \,\right\}.
$$
By definition, the isoperimetric inequality
$
    |\partial D| \geq I_X(|D|)
$
holds for all domains $D \subset X$, and the  isoperimetric profile is the largest function for which this inequality holds for all domains.

\smallskip 

Question~(ii) then consists in describing the structure of all domains $D$
satisfying
\[
   |\partial D| = I_X(|D|).
\]
In graph theory, one may consider the outer vertex boundary, the inner vertex boundary, or the edge boundary. In this article we focus on the inner vertex
boundary, for which the isoperimetric problem on the regular tree $T_d$ is non-trivial; by contrast, the outer and edge boundaries lead to essentially
trivial profiles (see Proposition~\ref{prop.outerboundary}).
Despite the extensive literature on isoperimetric inequalities on graphs, we are not aware of any previous work addressing the specific  inner vertex-isoperimetric problem on the infinite regular tree considered here.

\medskip

\textit{Remark.}  While edge-isoperimetric problems quantify the total amount of interaction
between a domain and its complement, the inner vertex boundary measures the
number of vertices within  the domain that are directly exposed to the exterior.
It is  the natural notion of boundary in vertex-based models, such as
epidemic spreading, information flow, or failure propagation in networked
systems.

\subsection{Outline of the paper}

In Section~\ref{sec.prelim} we introduce the basic graph-theoretic notions and notation used throughout the paper, and recall the relevant boundary notions for domains in graphs. 

\smallskip 

Specializing to the $d$-regular tree $T_d$, we introduce a natural decomposition
of the boundary of a domain into leaves and residual boundary vertices.
This leads to the notion of a \emph{full domain}, namely a domain whose entire
boundary consists of leaves.
The section concludes with a short proof of the classical identity relating the
size of a domain in $T_d$ to the size of its outer boundary:
\[
   |\partial' D| = (d-2)\,|D| + 2 .
\]

\smallskip 

The exact inner isoperimetric profile of the $d$-regular tree is established in Section~\ref{sec.mainthm}. It is given by
$$
   I_d(k) = \left\lceil \frac{(d-2)k + 2}{d-1} \right\rceil.
$$

\smallskip 

Section~\ref{sec.optimaldomains} investigates isoperimetrically optimal domains in $T_d$. Although one might initially expect metric balls to be the only optimal domains, B. L. Santos Correia observed in his thesis that the class of optimal domains is in fact considerably richer. In this section, we introduce a fundamental invariant, the \textit{boundary branching excess} $\tau(D)$, and show that optimality is characterized by the condition $\tau(D)\le d-2$. In particular, full domains satisfy $\tau(D)=0$ and are therefore optimal.

\smallskip 

Section~\ref{sec.fulldomains} is devoted to several equivalent characterizations
of full domains, and in Section~\ref{sec.gluing} we study a natural gluing
operation along boundary vertices for general domains and analyze the behavior
of the invariant $\tau$ under this operation.

\smallskip 

Section~\ref{sec.combstructure}, which forms one of the central parts of the paper, develops a detailed combinatorial description of arbitrary domains in
$T_d$.
We introduce the \emph{stem} of a domain, a finite tree that encodes how full components are attached along the residual boundary, establish its basic structural properties, and show how the invariant $\tau(D)$ can be read directly from its combinatorics. See in particular \S\ref{sec.stemdiagrams}
 
\smallskip

We further prove that any domain in $T_d$ can be reconstructed, up to isomorphism, from its stem together with a collection of full domains and suitable gluing data. This leads to a classification of all domains in the regular tree, and in particular of all isoperimetrically optimal domains, summarized in Theorem~\ref{th.reconstruction}.

\subsection{Related work}

Isoperimetric inequalities, and more generally geometric inequalities, have a long history dating back to the nineteenth century, and there is a vast literature which we do not attempt to survey here. A natural point of entry is the work of C.~Pittet~\cite{Pittet00}, which introduces the notion of isoperimetric profile from the perspective of Riemannian geometry and highlights the connections between continuous and discrete methods.

\smallskip 

Isoperimetric inequalities have also been extensively studied in geometric group theory. 
A fundamental result in this direction is that of Coulhon and Saloff-Coste~\cite{CoulhonSaloffCoste93}, who estimate the isoperimetric profile of a finitely generated group in terms of its growth function. 
See also~\cite{CT} for some recent refined estimates. 

Discrete isoperimetric problems also play a central role in combinatorics, expansion theory, percolation, random walk, and other probabilistic aspects of graph theory. General background can be found in~\cite{LyonsPeres2016, Woess2000}.

\smallskip 

The exact solution to the isoperimetric problem, including the exact isoperimetric profile and an explicit description of some (or all) minimal sets, is only known in a few cases. A classical example is that of the Hamming cube, whose vertex isoperimetric problem was solved by Harper~\cite{Harper64,Harper66}; see also Leader~\cite{Leader1991} for a pedagogical presentation. We further refer to Harper’s monograph~\cite{Harper2004} for a broad overview of combinatorial isoperimetric problems and the compression methods often used to study them.

\smallskip 

More closely related to the present work are recent studies of isoperimetric properties of finite trees, in particular the outer or edge isoperimetric profiles of complete $d$-ary trees, that is, balls of finite radius in the regular tree $T_d$. Relevant results and bounds can be found in \cite{BharadwajChandran2009, BonatoEtAl2024, Otachi, Vrto2010}.

\smallskip 

We also note that our notion of a \emph{full domain} is related to, but slightly different from, the notion of a \emph{full subtree} in the algorithms and data structures literature. In that setting, “full” typically refers to the number of children of each node in a rooted tree, whereas our notion is intrinsic to the underlying non-rooted tree and does not depend on a choice of root. On the other hand, our notion coincides with what is called a \emph{complete subtree} in the representation theory of groups acting on trees (see, e.g., \cite{Olshanskii1977}).

\smallskip

In this work, our focus is on the inner vertex boundary of domains in the infinite $d$-regular tree $T_d$. We obtain an exact formula for the corresponding isoperimetric profile together with a complete description of all optimal domains as assemblies of full components. To the best of our knowledge, this yields the first explicit and constructive classification of all inner vertex-isoperimetric minimizers in the regular tree $T_d$.

\section{Preliminaries}  \label{sec.prelim}

\subsection{Basic notions}

Throughout this paper we work with simple undirected graphs. We use standard graph-theoretic terminology and notation, as in
\cite{Bollobas,Diestel}. Let $X$ be a graph with vertex set $V(X)$ and edge set $E(X)$.
For a vertex $x\in V(X)$, we denote by $N_X(x)$ its set of neighbors, and by
\[
  \deg_X(x)=|N_X(x)|
\]
its degree in $X$. 
A nonempty subset $D\subset V(X)$ will always be identified with the induced subgraph on $D$; that is, we do not distinguish notationally between $D$ as a set of vertices and the corresponding induced subgraph. In particular, we write $|D|$ for the number of vertices of $D$.

\smallskip 

A \emph{domain} in $X$ is a finite connected subset $D\subset V(X)$.
If $|D|=2$, we call $D$ a \emph{two-vertex domain}.
If $D$ is a domain in $X$ and $x\in D$, the \emph{degree of $x$ in $D$} is defined by
\[
   \deg_D(x) = |N_X(x)\cap D|.
\]
A subset $S\subset V(X)$ is \emph{independent} if no edge of $X$ connects two vertices of $S$.

\smallskip

The \emph{Euler characteristic} of a finite graph $X$ is defined by
\[
   \chi(X)=|V(X)|-|E(X)|.
\]
For a connected finite graph, $\chi(X)=1$ if and only if $X$ is a finite tree, and for a forest $\chi(X)$ equals the number of connected components.

\smallskip

The \emph{inner vertex boundary} of a domain $D\subset X$ is defined by
\[
\partial D = \bigl\{\, x\in D \mid \exists\, y\notin D \text{ such that } \{x,y\}\in E(X) \,\bigr\}
= \bigl\{\, x\in D \mid \deg_D(x)<\deg_X(x) \,\bigr\}.
\]
Similarly, the \emph{outer vertex boundary} of $D$ is defined as
\[
\partial' D = \bigl\{\, y\notin D \mid \exists\, x\in D \text{ such that } \{x,y\}\in E(X) \,\bigr\},
\]
and the \emph{edge boundary} of $D$ is
\[
\partial_{\mathrm{edge}} D = \bigl\{\, \{u,v\}\in E(X) \mid  u\in D,\ v\notin D \,\bigr\}.
\]

We record two simple but useful facts on domains in a tree.

\begin{lemma}\label{lem:outer-boundary-tree}
Let $D$ be a domain in a tree $T$. Then every vertex $x\in \partial' D$ has exactly one neighbor in $D$.
Moreover, the outer boundary $\partial' D$ is an independent set.
\end{lemma}

\begin{proof}
By definition,  a vertex $x\in \partial' D$ has at least one neighbor in $D$.
If it had two distinct neighbors $y,z$ in $D$, then the unique path in $D$
joining $y$ to $z$, together with the edges $\{x,y\}$ and $\{z,x\}$ would form a cycle in $T$, contradicting the fact that
$T$ is a tree. Hence $x$ has exactly one neighbor in $D$.

\smallskip

Now suppose that two vertices $x,y \in \partial' D$ are adjacent. Let $u,v \in D$ be their respective neighbors in $D$. 
Then the unique path in $T$ joining $u$ to $v$ lies in $D$, and together with the edges $\{x,u\}$, $\{x,y\}$ and $\{y,v\}$ forms a cycle in $T$, again a contradiction. Therefore $\partial' D$ is an independent set.
\end{proof}

\medskip

\begin{lemma}\label{lem:embedding}
Every finite tree $G$ of maximum degree at most $d$ embeds  as a domain in $T_d$.
\end{lemma}

We include a brief proof for completeness.
\begin{proof}
Choose a root $x_0\in G$ and consider the rooted tree structure given by the distance from $x_0$. We construct the embedding inductively. 
Map $x_0$ to an arbitrary vertex $u_0\in T_d$. Suppose inductively that all vertices at distance at most $k$ from $x_0$ have been embedded in $T_d$ and pick a vertex $x\in G$ at level $k$. Its children in $G$ are adjacent only to $x$ and lie at level $k+1$.
Since $\deg_G(x)\le d$ and at most one neighbor of its image in $T_d$ has already been used (towards its parent),
there remain enough free neighbors to embed all its children.  
Since $G$ is finite, the construction ends after finitely many steps.
\end{proof}

\medskip

\subsection{Domains in a regular tree: leaves and residual boundary}

From now on, we work in the $d$-regular (or homogeneous) tree $T_d$, where $d\ge 2$. 
Recall that $T_d$ is the connected, acyclic graph in which every vertex has degree $d$. 
It is infinite and uniquely determined up to graph isomorphism. 
For general background on trees, we refer to \cite[Section~1.5]{Diestel} or \cite{West}. 

\medskip

A domain in $T_d$ is simply a finite subtree of $T_d$. Let $D\subset T_d$ be a domain with at least two vertices. 
We define its \emph{leaf boundary} and \emph{residual boundary} by
\[
   L(D)=\{x\in \partial D \mid \deg_D(x)=1\}, \qquad
   R(D)=\{x\in \partial D \mid   2 \leq  \deg_D(x) \leq d-1\}.
\]
Vertices in $L(D)$ are called \emph{leaves}, and $R(D)$ forms the \emph{residual boundary}. These two sets form a partition of the boundary: 
\[
   \partial D = L(D)\,\cup\, R(D), \quad L(D)\,\cap\, R(D) = \emptyset.
\]
The following elementary lemma on finite trees is well known and we state it without proof.
\begin{lemma}\label{lem:leaf-connected-complement}
Let $D\subset T_d$ be a domain with $|D|\ge 2$. Then $D$ contains at least two leaves.
Moreover, a vertex $x\in D$ is a leaf of $D$ if and only if $D\setminus\{x\}$ is connected.
\end{lemma}

\medskip 

\begin{definition}
A domain $D\subset T_d$  with $|D|\ge 2$ is said to have a \emph{full leaf boundary}, or simply to be a \emph{full domain},
if every boundary vertex is a leaf, that is, $\partial D = L(D)$; equivalently $R(D)=\emptyset$.
\end{definition}
Full domains will play a central role in the description of optimal domains developed later.
We end this section by stating a basic property of the outer boundary of a domain in $T_d$.
\begin{proposition}\label{prop.outerboundary}
The boundary of an arbitrary domain $D\subset T_d $ satisfies
\begin{equation}\label{eq.outerboundary}
   |\partial' D|  = (d-2)\,|D| + 2  = \sum_{x\in\partial D} (d-\deg_D(x)).
\end{equation}
\end{proposition}

\smallskip 

Although this identity is probably well known to experts and appears in \cite{CT}, we include a short proof, since the argument provides a useful template for several constructions that follow in the present paper.

\smallskip 

\begin{proof}
We argue by induction on $|D|$.  If $|D|=1$, then $\partial' D$ consists of the $d$ neighbors of the unique vertex of $D$, and the formula holds. Assume $|D|\ge 2$, and choose a leaf $x\in\partial D$.  
Then $x$ has one neighbor inside $D$ and $d-1$ neighbors outside.  
Let $D^{-}=D\setminus\{x\}$. 
Removing $x$ deletes its $d-1$ neighbors outside of $D$ from $\partial'D$, while $x$ itself becomes a vertex of $\partial' D^{-}$.
$$
    |\partial' D^{-}| = |\partial' D| - (d-1) + 1  = |\partial' D| - (d-2).
$$
On the other hand, by Lemma \ref{lem:leaf-connected-complement},  $D^{-}$ is connected, so we have from  the induction hypothesis:
$$
   |\partial' D^{-}|  =  (d-2)\,|D^{-}| + 2 = (d-2)(|D|-1) + 2.
$$
We thus conclude that 
$$
     |\partial' D|  =   |\partial' D^{-}| + (d-2)  = (d-2)(|D|-1) + 2 +(d-2)  = (d-2) |D|+ 2,
$$
proving the first equality in \eqref{eq.outerboundary}. 
For the second equality, observe that for each $x\in\partial D$, exactly $d-\deg_D(x)$ of its $d$ neighbors lie outside $D$.  
Summing over $x\in\partial D$ gives
\[
   |\partial' D| = \sum_{x\in\partial D} (d-\deg_D(x)).
\]
\end{proof}

\medskip

\begin{remark}
(i) In a regular tree, the outer vertex boundary $\partial' D$ and the edge boundary $\partial_{\mathrm{edge}} D$ are naturally in bijection. In particular, \eqref{eq.outerboundary} remains valid with $|\partial_{\mathrm{edge}} D|$ replacing $|\partial' D|$.

\smallskip

(ii) In the literature on finite trees, the term \emph{interior vertex} often refers to a vertex that is not a leaf.
In the present paper, since we work with domains $D$ in a tree, an interior vertex would more naturally be defined as a vertex all of whose neighbors belong to $D$. To avoid any ambiguity, we altogether avoid the terminology of interior vertex here. 
\end{remark}

\section{The isoperimetric profile of the regular tree} \label{sec.mainthm}

In this section, we compute the isoperimetric profile $I_d$ of the $d$-regular tree $T_d$. Recall that this is the function $I_d : \mathbb{N} \to \mathbb{N}$ defined by
\[
I_d(k) = \min \bigl\{ |\partial D| \,\big|\, D \subset T_d \text{ is finite, connected, and } |D| = k \bigr\},
\]
for all $k \in \mathbb{N}$.
Observe that $I_d(k) \leq k$  since $\partial D \subset D$. 
\begin{theorem}\label{thm:Id}
For every $d\ge 2$, the isoperimetric profile of $T_d$ is given by $I_d(1)=1$, and 
\begin{equation}\label{eq.isopofile}
   I_d(k)=\Bigl\lceil \frac{(d-2)k+2}{d-1}\Bigr\rceil,
\end{equation}
for $k\ge 2$.
\end{theorem}
Here $\lceil x \rceil$ is the ceiling function, that is  $m = \lceil x \rceil$ is the unique integer such that  $m-1 < x \leq m$.

\medskip 

The formula \eqref{eq.isopofile} already appears in \cite{CorreiaThesis2025}. For the convenience of the reader, we include here a self-contained proof, which relies on the following three lemmas.

\medskip 

\begin{lemma}\label{lem:Id-lower}
For every finite connected $D\subset T_d$ with $|D|\ge 2$ we have
\[
   |\partial D|\  \geq \frac{(d-2)|D|+2}{d-1}.
\]
\end{lemma}

\begin{proof}
Each boundary vertex has at most $(d-1)$ neighbors outside $D$, so 
\[
     |\partial' D|  \leq (d-1)|\partial D|.
\]
Combining this with Proposition~\ref{prop.outerboundary}, which states that $|\partial'D| =  (d-2)|D|+2$, yields the claim.
\end{proof}

The next two lemmas are purely arithmetical: 
\begin{lemma}\label{lem:smallk}
For $k\ge 2$ we have
\[
   k=\Bigl\lceil \frac{(d-2)k+2}{d-1}\Bigr\rceil  \quad\Longleftrightarrow\quad    2\le k\le d.
\]
\end{lemma}

\begin{proof}
A direct computation gives
\[
\frac{(d-2)k+2}{d-1}=k-\frac{k-2}{d-1},
\]
which lies in the interval $(k-1,k]$ precisely when $0\le k-2<d-1$, i.e.\ when $2\le k\le d$.
\end{proof}

\medskip 

\begin{lemma}\label{lem:division}
Let $d\ge 2$ and $k\ge d+1$, and set
\[
   s = \Bigl\lceil \frac{k-2}{d-1}\Bigr\rceil  \qquad \text{and} \qquad    q = (d-1)s-(k-2).
\]
Then $s\ge 1$ and $0\le q<d-1$.
\end{lemma}

Note that this  lemma expresses the Euclidean division of $2-k$ by $d-1$, with remainder $q$ and quotient~$-s$.

\begin{proof}
By definition of $s$, we have
\[
s-1<\frac{k-2}{d-1}\le s.
\]
Since $k\ge d+1$, it follows that $\frac{k-2}{d-1}\ge 1$, hence $s\ge 1$. Multiplying the above inequalities by $d-1$ yields
\[
(d-1)(s-1)< k-2 \le (d-1)s,
\]
which can be rewritten as
\[
0\le q=(d-1)s-(k-2)<d-1.
\]
\end{proof}

\medskip 

The remainder of this section is devoted to the proof of Theorem~\ref{thm:Id}.

\begin{proof} 
The result is trivial for  $k=1$. The case $d = 2$ is also trivial since in that case both sides of  \eqref{eq.isopofile} are clearly equal to $2$ (when $k \geq 2$). For the rest of the proof we thus assume $d\geq 3$ and $k\ge 2$. 

\medskip 

 Lemma~\ref{lem:Id-lower} gives us 
\[
   I_d(k) \  \geq  \frac{(d-2)k+2}{d-1},
\]
and since the left-hand side is an integer, we have in fact
\[
   I_d(k) \  \geq  \Bigl\lceil \frac{(d-2)k+2}{d-1}\Bigr\rceil.
\]
To prove equality, we need to produce for each $k\ge 2$, an explicit domain $D\subset T_d$ with $|D|=k$ and
\begin{equation}\label{domoptimal1}
    |\partial D| = \Bigl\lceil \frac{(d-2)k+2}{d-1}\Bigr\rceil.
\end{equation}
We divide the construction  in the following three cases:
\begin{enumerate}[(i)]
\item $2\le k\le d$,
\item $k\geq d+1$ and  $k \equiv 2  \mod (d-1)$.
\item $k\geq d+1$ and  $k \not\equiv 2  \mod (d-1)$.
\end{enumerate}

Case (i): \  If $2\le k\le d$, any connected set $D$ of cardinality $k$ satisfies
$\partial D = D$.  Using  Lemma~\ref{lem:smallk}, we then have 
$$|\partial D| = k = \Bigl\lceil \frac{(d-2)k+2}{d-1}\Bigr\rceil.$$

\medskip 

Case (ii): \  We now assume  $k\geq d+1$ ($\geq 3$) and  $k \equiv 2  \mod (d-1)$, i.e. there exists an integer $s\geq 0$ such that $k-2 = (d-1)s$.
If  $s=0$, then $k=2$ and this case has been considered in the previous step. We therefore assume $s\geq 1$
and we consider the  domain $D\subset T_d$ defined as follows: 
$$
  D = A_s \cup \partial'A_s,
$$
where  $A_s = \{x_1,\dots,x_s\}$ is a path in $T_d$ (i.e. $x_i$ is adjacent to $x_{i+1}$ for $1\leq i \leq s-1$)  and $\partial'A_s$ is its outer boundary in $T_d$. A direct inspection shows that
$$
 \partial D = \partial'A_s, \quad   |\partial D| = |\partial'A_s| =  (d-2)s + 2,
$$
and
$$
  D =  A_s \cup \partial'A_s, \quad  k :=  |D| = s + (d-2)s + 2 = (d-1)s + 2.
$$
A   straightforward computation yields
$$
   \frac{(d-2)k+2}{d-1} = (d-2)s + 2,
$$
which is an integer, so
$$
   |\partial D| = (d-2)s + 2  = \frac{(d-2)k+2}{d-1}  = \Bigl\lceil \frac{(d-2)k+2}{d-1}\Bigr\rceil.
$$

\medskip 

Case (iii): \  We finally  assume  $k\geq d+1$ and  $k \not\equiv 2  \mod (d-1)$. By Lemma~\ref{lem:division} we can write
\begin{equation}\label{eq.k-division}
   k-2 = (d-1)s - q,
\end{equation}
with integers $s\ge 1$ and $1\le q<d-1$. 
In that case we consider the domain $D$ obtained by a modification of the previous construction. Again, let 
$A_s = \{x_1,\dots,x_s\}$ be a  path in $T_d$ and define now:
$$
 D = A_s\cup \partial' A_s\setminus Q,
$$
 where 
$Q$ is a set of $q$ external neighbors of $x_s$ (see Figure \ref{fig:example-D}). We then have
$$
   |D|    =  (d-1)s  + 2 - q   =:k.
$$
and
$$
   |\partial D| = |\partial'A_s| -q +1 =  (d-2)s + 3-q,
$$
since we have removed $q$ leaves compared to the previous construction, but now there is an additional boundary vertex  $x_s \in \partial D$. 

\smallskip 

Using  $k = (d-1)s +2- q$ and  $|\partial D| = (d-2)s - q + 3$,  we compute
\begin{align*}
 \frac{(d-2)k+2}{d-1} &=\frac{(d-2)\bigl((d-1)s+2-q\bigr)+2}{d-1} =(d-2)s+\frac{2(d-1)-(d-2)q}{d-1}
\\ &=  |\partial D| -1+\frac{q}{d-1},
\end{align*}
and since $1\le q<d-1$, we have $0< \frac{q}{d-1}<1$, therefore 
\[
   |\partial D| - 1  < \frac{(d-2)k+2}{d-1}  < |\partial D|;
\]
equivalently 
\[
   \Bigl\lceil \frac{(d-2)k+2}{d-1}\Bigr\rceil  = |\partial D|.
\]
In all cases we have constructed a domain $D$ with $|D|=k$ and $|\partial D| = \bigl\lceil \frac{(d-2)k+2}{d-1}\bigr\rceil$;  this completes the proof.
\end{proof}

\medskip 

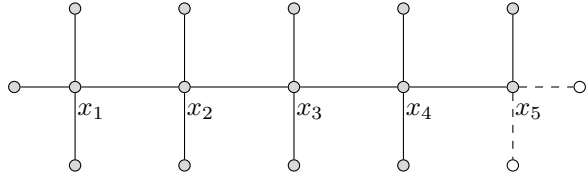
\begin{figure}[h]
\centering
\begin{tikzpicture}[
  scale=0.8,
  inD/.style={circle,draw,fill=black!15,inner sep=1.5pt},
  outD/.style={circle,draw,fill=white,inner sep=1.5pt},
  every node/.style={font=\small}
]

\node[inD, label={[xshift=6pt]  below  :$x_1$}] (x1) at (0,0) {};
\node[inD, label={[xshift=6pt]below:$x_2$}] (x2) at (1.8,0) {};
\node[inD, label={[xshift=6pt]below:$x_3$}] (x3) at (3.6,0) {};
\node[inD, label={[xshift=6pt]below:$x_4$}] (x4) at (5.4,0) {};
\node[inD, label={[xshift=6pt]below:$x_5$}] (x5) at (7.2,0) {};

\node[inD] (a1) at (0,1.3) {};
\node[inD] (a2) at (0,-1.3) {};
\node[inD] (a3) at (-1,0)     {};

\node[inD] (b1) at (1.8,1.3) {};
\node[inD] (b2) at (1.8,-1.3) {};

\node[inD] (c1) at (3.6,1.3) {};
\node[inD] (c2) at (3.6,-1.3) {};

\node[inD] (d1) at (5.4,1.3) {};
\node[inD] (d2) at (5.4,-1.3) {};

\node[inD]  (e1) at (7.2,1.3) {};    
\node[outD] (m1) at (7.2,-1.3) {};   
\node[outD] (m2) at (8.3,0)    {};   

\draw (x1) -- (x2) -- (x3) -- (x4) -- (x5);

\draw (x1) -- (a1);
\draw (x1) -- (a2);
\draw (x1) -- (a3);

\draw (x2) -- (b1);
\draw (x2) -- (b2);

\draw (x3) -- (c1);
\draw (x3) -- (c2);

\draw (x4) -- (d1);
\draw (x4) -- (d2);

\draw (x5) -- (e1);
\draw[dashed] (x5) -- (m1);
\draw[dashed] (x5) -- (m2);
\end{tikzpicture}
\caption{\small
The domain $D$ constructed in the last step of the proof, for the case $d = 4$, $s = 5$ and $q = 2$. 
(Thus $|D| = 15$ and $|\partial D| = I_4(15) = 11$.)
}
\label{fig:example-D}
\end{figure}

\subsubsection*{On the Cheeger constant}\label{subsec.cheeger}

The \emph{inner Cheeger constant} of an infinite graph $X$ is defined as
\[
 h_{\mathrm{in}}(X) = \inf\left\{    \frac{|\partial D|}{|D|}    \;\middle|\;    D\subset X \text{ finite, connected, and non-empty}
 \right\}
 = \inf_{k\in \mathbb{N}} \;  \frac{I_X(k)}{k}.
\]
For the regular tree  $X = T_d$,  we obtain the following value:
$$  
h_{\mathrm{in}}(T_d) = \frac{d-2}{d-1}.
$$
Indeed, applying Theorem~\ref{thm:Id}, we obtain  
$$
 \frac{I_{T_d}(k)}{k} = \frac{1}{k} \Bigl\lceil \frac{(d-2)k+2}{d-1}\Bigr\rceil.
$$
Therefore
$$
\frac{d-2}{d-1} \le \frac{I_{T_d}(k)}{k} < \frac{d-2}{d-1} + \frac{2}{(d-1)k} + \frac{1}{k}.
$$
It follows that
$$
 \frac{d-2}{d-1} \leq h_{\mathrm{in}}(T_d) = \inf_{k\in\mathbb N}\frac{I_{T_d}(k)}{k}
 \leq \lim_{k\to\infty}\frac{I_{T_d}(k)}{k} = \frac{d-2}{d-1},
$$
which proves the claim.

 \medskip 

Using the outer and edge  boundaries, we similarly define  the \textit{outer} and \textit{edge} Cheeger constants $h_{\mathrm{out}}(X)$ and $h_{\mathrm{edge}}(X)$ (the latter is also referred to as the \textit{expansion constant} of the graph).  Applying Proposition~\ref{prop.outerboundary}, we see that 
$$
   h_{\mathrm{edge}}(T_d)  = h_{\mathrm{out}}(T_d) = d-2.
$$
This value is  classical, it appears e.g. as Exercise 6.1 in \cite{LyonsPeres2016}.
We are not aware of an explicit reference for the above value of $h_{\mathrm{in}}(T_d)$ in the literature.

\section{Optimal domains}
\label{sec.optimaldomains}

A domain $D\subset T_d$ is said to be \emph{isoperimetrically optimal}, or simply an \emph{optimal domain},
if it has minimal boundary among all domains of the same size, that is,
\[
|\partial D| = I_d(|D|).
\]

As a first example, consider the ball $B(x,r)\subset T_d$ of radius $r\ge 2$ centered at some vertex $x$. 

A straightforward induction shows that
\[
|\partial B(x,r)| = d(d-1)^{r-1}, \qquad |B(x,r)| = 1 + d\,\frac{(d-1)^r - 1}{d-2}.
\]
One readily checks that
\[
|\partial B(x,r)| = \frac{(d-2)|B(x,r)| + 2}{d-1} = I_d\left( |B(x,r)| \right), 
\]
therefore the ball $B(x,r)$ is an optimal domain; but, as we shall see now, the class of optimal domains is much richer than the family of balls. To characterize them, we introduce the following invariant:
\begin{definition}
The \emph{boundary branching excess} of a domain $D\subset T_d$ of cardinality at least 2 is defined as 
\begin{equation}\label{def.tau}
  \tau(D) = \sum_{x\in\partial D} \bigl(\deg_D(x)-1\bigr).
\end{equation}
Observe that in fact 
\begin{equation}\label{def.taubis}
  \tau(D) = \sum_{x\in R(D)} \bigl(\deg_D(x)-1\bigr),
\end{equation}
since $\deg_D(x) = 1$ for any $x\in L(D)$.
\end{definition}

\medskip 

\begin{lemma}\label{lem.tau-formula}
Let $D\subset T_d$ be a finite  domain with $|D| \geq 2$. Then
\begin{equation}\label{eq.tau-formula}
  \tau(D)  = (d-1)\,|\partial D| - (d-2)\,|D| - 2.
\end{equation}
Furthermore 
\begin{equation}\label{inequtau1}
  |R(D)| \leq  \tau(D)  \leq |D| -2.
\end{equation}
Moreover, the equality $\tau(D)  =  |D| -2$ holds  if and only if $\partial D=D$.
\end{lemma}

\smallskip 

Observe that in \eqref{eq.tau-formula}, for domains of fixed cardinality  $k \geq 2$ the boundary size $|\partial D|$ is  minimized exactly when $\tau(D)$ is minimized. In other words, isoperimetric optimality is equivalent to minimality of $\tau(D)$ among domains of the same size.

\begin{proof}   Every boundary vertex $x\in\partial D$ has $d-\deg_D(x)$
neighbors outside $D$, so
\[
  |\partial' D| = \sum_{x\in\partial D} (d-\deg_D(x))
    = \sum_{x\in\partial D} \bigl[(d-1) - (\deg_D(x)-1)\bigr]
    = (d-1)\,|\partial D| - \tau(D),
\]
hence \eqref{eq.tau-formula} follows from Proposition \ref{prop.outerboundary}.
The inequality $|R(D)| \leq  \tau(D)$ follows immediately from \eqref{def.taubis} since each term in the sum is $\geq1$.  To prove the inequality   $\tau(D)  \leq |D| -2$, we use that $ \chi(D) = 1$, hence 
$$
 \sum_{x\in D}\deg_D(x)=2|E(D)|=2(|D|-1);
$$
therefore  
$$
 \tau(D)   = \sum_{x\in \partial D} (\deg_D(x)-1) \leq \sum_{x\in D}\bigl(\deg_D(x)-1\bigr) =2(|D|-1)-|D| =|D|-2,
$$
with equality  if and only if $\partial D=D$.
\end{proof}

\medskip 

Using the boundary branching excess, we state a convenient optimality criterion: 
\begin{theorem}\label{CritereIsop}
A domain $D\subset T_d$ with $|D|\ge 2$ is isoperimetrically optimal,   i.e.\ $|\partial D|=I_d(|D|)$,  if and only if
\begin{equation}\label{ineq:tau-optimal}
   \tau(D) \leq  d-2.
\end{equation}
\end{theorem}

\medskip 

\begin{proof}
By Theorem~\ref{thm:Id}, the domain $D$ is optimal if and only if
\[
 |\partial D| = \lceil A\rceil,
 \quad \text{where} \quad
 A := \frac{(d-2)|D|+2}{d-1}.
\]
By Lemma~\ref{lem.tau-formula}, we have
\[
  |\partial D|   =   A + \frac{\tau(D)}{d-1}.
\]
Since $\tau(D)\ge 0$, we have $|\partial D|\ge A$, and hence
$|\partial D|=\lceil A\rceil$ if and only if
\[
 0 \le \frac{\tau(D)}{d-1} < 1.
\]
As $\tau(D)$ is a non-negative integer, this is equivalent to
$\tau(D)\le d-2$, which completes the proof.
\end{proof}

\medskip  

Our next result provides some simple examples of optimal domains: 
\begin{corollary}\label{CorcritereIsop}
\begin{enumerate}[(a)]
\item Every domain of size $k\le d$ is optimal.
\item Every full domain is optimal. 
\item If  $|R(D)| = 1$, then $D$ is optimal.
\item If $D$ is optimal, then $|R(D)|\le d-2$.
\end{enumerate}
\end{corollary}

\medskip 

\begin{proof}
(a) That every domain of size $k\le d$ is optimal has been established in step (i) of the proof  of Theorem~\ref{thm:Id}.
It also follows from  the previous Theorem together with Lemma \ref{lem.tau-formula}. 

\medskip 

(b) For a full domain $D$,  we have $R(D) = \emptyset$ and hence $\tau(D)=0$ 
by \eqref{def.taubis}. Therefore $D$ is optimal.

\medskip 

(c)  If $R(D)=\{x\}$, then $\tau(D)=\deg_D(x)-1\le d-2$.

\medskip 

(d) Combining \eqref{inequtau1} and \eqref{ineq:tau-optimal}, we see if $D$ is an optimal domain, then  $|R(D)|\leq  \tau(D) \leq  d-2$.
\end{proof}

\bigskip 

We now describe a more elaborate example. 

\medskip

Consider the domain $D \subset T_5$  in Figure~\ref{figaa} below. The boundary of $D$ consists of   $12$ leaves and the residual vertex $p$, which is of degree $3$. Hence $R(D) = \{p\}$ and $\tau(D) = 2 \leq 5-2$ hence  $D$ is an optimal domain.
This can also be verified by direct computations: we see from the picture that  $|D| = 16$ and $|\partial D| = 13$. With $d=5$ and $k=16$, Theorem~\ref{thm:Id} gives
\[
   I_5(|D|) = I_5(16) = \left\lceil \frac{(5-2)\cdot 16 + 2}{5-1} \right\rceil
   = \left\lceil \frac{50}{4} \right\rceil
   = 13 = |\partial D|, 
\]
confirming  that $D$ is indeed optimal.

\medskip

Now consider the domain $D \subset T_5$ in Figure~\ref{figab}, which can be viewed as a modification of the previous example where one vertex has been moved. The residual boundary now consists of two points $p,q$, each of degree $4$.
Hence $\tau(D) = 3+3 = 6 \not\leq 5-2$.  This shows that $D$ cannot be optimal.

Again, this can be seen by directly counting the vertices. We have again $|D| = 16$ but now  $|\partial D| = 14$,
which is strictly larger than $I_5(16) = 13$. 

\medskip

\begin{figure}[H]
\centering
\begin{minipage}{0.48\textwidth}
\centering
\begin{tikzpicture}[scale=0.8,
  vertex/.style={circle,fill=black,inner sep=1.2pt},
  every path/.style={thick}
]
\node[vertex,label=below:$p$] (p) at (0,0) {};

\node[vertex] (c1) at (2,0) {};
\node[vertex] (c2) at (-1,1.7) {};
\node[vertex] (c3) at (-1,-1.7) {};

\draw (p)--(c1) (p)--(c2) (p)--(c3);

\node[vertex] (a11) at (3,0) {};
\node[vertex] (a12) at (2.7,0.9) {};
\node[vertex] (a13) at (2.7,-0.9) {};
\node[vertex] (a14) at (1.2,1) {};
\draw (c1)--(a11) (c1)--(a12) (c1)--(a13) (c1)--(a14);

\node[vertex] (a21) at (-2,2.2) {};
\node[vertex] (a22) at (-0.3,2.7) {};
\node[vertex] (a23) at (-1.8,0.9) {};
\node[vertex] (a24) at (0,1.1) {};
\draw (c2)--(a21) (c2)--(a22) (c2)--(a23) (c2)--(a24);

\node[vertex] (a31) at (-2,-2.2) {};
\node[vertex] (a32) at (-0.3,-2.7) {};
\node[vertex] (a33) at (-1.8,-0.9) {};
\node[vertex] (a34) at (0,-1.1) {};
\draw (c3)--(a31) (c3)--(a32) (c3)--(a33) (c3)--(a34);
\end{tikzpicture}

\caption{An optimal domain in $T_5$ with\\ $|D|=16$, $|\partial D|=13$, and $\tau(D)=2$.}
\label{figaa}
\end{minipage}
\hfill
\begin{minipage}{0.48\textwidth}
\centering
\begin{tikzpicture}[scale=0.8,
  vertex/.style={circle,fill=black,inner sep=1.2pt},
  every path/.style={thick}
]

\node[vertex] (c1) at (2,0) {};
\node[vertex] (c2) at (-1,1.7) {};
\node[vertex] (c3) at (-1,-1.7) {};

\draw (p)--(c1) (p)--(c2) (p)--(c3);

\node[vertex] (a23) at (-1.8,0.9) {};

\draw (p)--(a23);

\node[vertex] (a11) at (3,0) {};
\node[vertex] (a12) at (2.7,0.9) {};
\node[vertex] (a13) at (2.7,-0.9) {};
\node[vertex] (a14) at (1.2,1) {};
\draw (c1)--(a11) (c1)--(a12) (c1)--(a13) (c1)--(a14);

\node[vertex] (a21) at (-2,2.2) {};
\node[vertex] (a22) at (-0.3,2.7) {};
\node[vertex] (a24) at (0,1.1) {};
\draw (c2)--(a21) (c2)--(a22) (c2)--(a24);

\node[vertex] (a31) at (-2,-2.2) {};
\node[vertex] (a32) at (-0.3,-2.7) {};
\node[vertex] (a33) at (-1.8,-0.9) {};
\node[vertex] (a34) at (0,-1.1) {};
\draw (c3)--(a31) (c3)--(a32) (c3)--(a33) (c3)--(a34);

\node[vertex,label=below:$p$] (p) at (0,0) {};
\node[vertex,label=right:$q$] at (c2) {};

\end{tikzpicture}
\caption{A non-optimal domain in $T_5$ \\ with $|D|=16$, $|\partial D|=14$, and $\tau(D)=6$.}
\label{figab}
\end{minipage}
\end{figure}

\section{On full domains}
\label{sec.fulldomains}

Recall that a domain $D\subset T_d$ is full  if every boundary vertex is a leaf. Full domains are isoperimetrically optimal and play a central role in the sequel. In this section we give several equivalent characterizations in the case $|D|\ge 3$.

\begin{theorem}
Let $D\subset T_d$ be a domain with $|D|\ge 3$ and $d\ge 2$.  
The following conditions are equivalent:
\begin{enumerate}[(a)]
\item $D$ is full, i.e.\ $\partial D = L(D)$;
\item $R(D)=\emptyset$;
\item $\tau(D)=0$;
\item $|\partial D|  =   \dfrac{(d-2)\,|D| + 2}{d-1}$;
\item $D$ is optimal and $|D| \equiv 2 \pmod{d-1}$;
\item $D\setminus \partial D$ is connected;
\item there exists a domain $D_1\subset D$ such that  
      $D = D_1 \cup \partial' D_1$.
\end{enumerate}
\end{theorem}

\medskip 

\begin{remark}
The hypothesis $|D|\ge 3$ cannot be removed.  
When $|D|=2$, the domain $D$ is always full, so conditions (a)--(f) hold trivially, but condition (g) fails.
Note also that the theorem is trivially true when $d=2$. In fact every domain  in $T_2$ is full. 
\end{remark}

\medskip

\begin{proof} 
The equivalences (a) $\Leftrightarrow$ (b) $\Leftrightarrow$ (c) follow directly from the definitions of $L(D)$, and $\tau(D)$; and 
(c)  $\Leftrightarrow$ (d) follows from the identity \eqref{eq.tau-formula}, which says that 
\begin{equation}\label{eq.boundary-expression2}
    |\partial D| = \frac{(d-2)|D|+2}{d-1} + \frac{\tau(D)}{d-1}.
\end{equation}
The following implications remain to be proved:
$$
 (d)   \Rightarrow  (e)  \Rightarrow  (c),   \qquad    (a) \Leftrightarrow  (f) ,      \qquad      (a)\Leftrightarrow  (g).
$$

\smallskip

(d) $\Rightarrow$ (e).  
If (d) holds then, by Theorem~\ref{thm:Id}, we have $|\partial D| = I_d(|D|)$, so $D$ is optimal.  
Moreover  since  $(d-2)\equiv -1 \pmod{d-1}$, we have
$$
   -|D|+2 \equiv (d-2)|D|+2 \equiv (d-1) |\partial D|  \equiv 0 \pmod{d-1}.
$$
Hence $|D| \equiv 2 \pmod{d-1}$.

\smallskip

(e) $\Rightarrow$ (c).  
Using \eqref{eq.tau-formula} or \eqref{eq.boundary-expression2}, and computing modulo $(d-1)$, we have
\[
\tau(D)\equiv -(d-2)|D|-2 \equiv |D|-2 \pmod{d-1}.
\]
Thus the hypothesis $|D| \equiv 2 \pmod{d-1}$ implies $\tau(D)\equiv 0\pmod{d-1}$.  On the other hand, since $D$  is optimal, Theorem~\ref{CritereIsop} gives $0\le \tau(D)\le d-2$. It follows that $\tau(D)=0$.

\smallskip

(a) $\Leftrightarrow$ (f).  
By Lemma~\ref{lem:leaf-connected-complement}, removing any leaf of a domain does not disconnect the remainder.  
Applying this repeatedly shows that $D\setminus \partial D$ is connected if and only if  all boundary vertices are leaves.

\smallskip

(g) $\Rightarrow$ (a).
Assume that $D = D_1 \cup \partial' D_1$ for some domain $D_1$.
Then $\partial D = \partial' D_1$, and by Lemma~\ref{lem:outer-boundary-tree},
each vertex $x\in\partial D$ has exactly one neighbor in $D$.
Thus every boundary vertex is a leaf, i.e.\ $\partial D = L(D)$.

\smallskip

Finally we prove (a) $\Rightarrow$ (g). Assume $\partial D = L(D)$ and set
\[
 D^{\mathrm{o} }  = D\setminus\partial D = D\setminus L(D).
\]
Since $|D|\ge 3$ the set $D^{\mathrm{o} } $ is non-empty and by (f) it is connected. We now prove  that $\partial'D^{\mathrm{o} }  = \partial D$.

\smallskip

We first show that $\partial' D^{\mathrm{o} }  \subset \partial D$.  Assume by contradiction that there exists a vertex $z \in \partial' D^{\mathrm{o} }  \setminus \partial D$.
Then $z \notin D$, and since $z \in \partial' D^{\mathrm{o} } $, it is adjacent to some vertex $y \in D^{\mathrm{o} } $.
As $D^{\mathrm{o} }  \subset D$, this implies that $y \in D$ has a neighbor $z \notin D$, hence $y \in \partial D$. However, by definition $D^{\mathrm{o} }  = D \setminus \partial D$,
so $y \notin \partial D$, a contradiction. Therefore, no such vertex $z$ exists, and we conclude that $\partial' D^{\mathrm{o} }  \subset \partial D$.

\smallskip

To prove the converse inclusion $\partial D \subset \partial' D^{\mathrm{o} } $, choose a vertex $x\in\partial D$.
Since $x$ is a leaf of $D$, it has exactly one neighbor $y\in D$; two leaves cannot be adjacent, so $y\notin L(D)$ and therefore $y\in D^{\mathrm{o} } $. Thus $x\in\partial'D^{\mathrm{o} } $.
This shows that $\partial' D^{\mathrm{o}} = \partial D$. 
Since $D^{\mathrm{o}} = D\setminus \partial D$, we obtain
\[
D = D^{\mathrm{o}} \cup \partial D
  = D^{\mathrm{o}} \cup \partial' D^{\mathrm{o}},
\]
so that (g) holds with $D_1 = D^{\mathrm{o}}$. 
This completes the proof.
\end{proof}

\medskip

As an immediate consequence of (e), we have the following 
\begin{corollary}  
Let $D$ be a domain in $T_d$ such that $|D| \equiv 2 \pmod{d-1}$. Then $D$ is optimal if and only if $D$ is full.
\end{corollary}

\section{Gluing domains along a vertex}
\label{sec.gluing}

We now introduce a basic operation for assembling larger domains from smaller ones.
\begin{definition}
Let $D_1,\dots,D_m$ be domains in $T_d$, each of cardinality at least~$2$,  with $2 \leq m\leq d$, and let
$p\in T_d$ be a vertex such that $p\in \partial D_i$ for all $i$.  
We say that the domain $D$ is obtained by \emph{gluing} $D_1,\dots,D_m$ at the vertex ~$p$ if
\begin{enumerate}[(i)]
\item $D=\bigcup_{i=1}^m D_i$,
\item $\partial D_i\cap \partial D_j=\{p\}$ for all $i\ne j$.
\end{enumerate}
\end{definition}

Geometrically, we attach $m$ subtrees along the single common vertex~$p$.  From the definition it follows immediately that
\begin{equation}\label{eq.degree-sum}
   \deg_D(p)=\sum_{i=1}^m \deg_{D_i}(p),
\end{equation}
thus $p\in\partial D$ if and only if
$$
   \sum_{i=1}^m \deg_{D_i}(p)<d.
$$
\begin{example}
If $m=2$ and $D_2$ is a two-vertex domain, then gluing $D_2$ to another domain $D_1$ amounts to attaching a leaf to $D_1$ at a vertex $p\in \partial D_1$.
\end{example}

\medskip

\begin{proposition}\label{prop:tau-gluing}
Let $D$ be obtained by gluing the domains $D_1,\dots,D_m$ at~$p$. Then
\[
   \tau(D) = \begin{cases}
       \displaystyle \sum_{i=1}^m \tau(D_i) + (m-1),          &\text{if } p\in\partial D, \\[0.3cm]
       \displaystyle \sum_{i=1}^m \tau(D_i) - (d-m),           &\text{if } p\notin\partial D .
     \end{cases}
\]
In particular, gluing increases the total branching excess  $\tau$ by exactly  $m-1$ if  $p\in \partial D$, while $\tau$ remains unchanged or decreases if  $p \not\in \partial D$.
\end{proposition}

\begin{proof}
For each $i$ and any vertex  $x\ne p$ lying in $D_i$ we have $\deg_D(x)=\deg_{D_i}(x)$.  As for the point $p$,  from \eqref{eq.degree-sum} we have
\begin{equation}\label{eq.deg-expand}
   \deg_D(p)-1  = \left( \sum_{i=1}^m \deg_{D_i}(p)\right) - 1    = \sum_{i=1}^m (\deg_{D_i}(p)-1) + (m-1).
\end{equation}
We now separately consider the cases where $p$ belongs to the boundary of $D$ or not.

\smallskip 

If  $p\in\partial D$, then  using~\eqref{eq.deg-expand}, we obtain
\begin{align*}
\tau(D)  &= \sum_{x\in \partial D\setminus \{p\}} \bigl(\deg_D(x)-1\bigr)    + \bigl(\deg_D(p)-1\bigr) 
 \\[0.3em] &=  \sum_{i=1}^m    \left(\sum_{x\in \partial D_i\setminus \{p\}}    \bigl(\deg_{D_i}(x)-1\bigr)\right)    + \bigl(\deg_D(p)-1\bigr) \\[0.3em]
&= \sum_{i=1}^m    \left(\sum_{x\in \partial D_i}    \bigl(\deg_{D_i}(x)-1\bigr)\right)    + (m-1)  = \sum_{i=1}^m \tau(D_i) + (m-1).
\end{align*}

\smallskip 

If $p\notin\partial D$, the argument is similar, except that the contribution $\deg_D(p)-1$ must \textit{not} be counted in the computation of $\tau(D)$. Since in this case $\deg_D(p)=d$, we obtain
\[
 \tau(D) = \sum_{i=1}^m \tau(D_i) + (m-1) - (d-1)  = \sum_{i=1}^m \tau(D_i) - (d-m).
\]
This concludes the proof.
\end{proof}

\medskip

Combining Theorem~\ref{CritereIsop} with Proposition~\ref{prop:tau-gluing} we obtain the following optimality criterion.
\begin{corollary}\label{cor:gluing-optimal}
Let $D$ be obtained by gluing the domains $D_1,\dots,D_m$ at a vertex~$p$.
Then $D$ is isoperimetrically optimal if and only if one of the following conditions holds:
\begin{enumerate}[(i)]
\item If $p\in\partial D$, then
      \[
         \sum_{i=1}^m \tau(D_i)+(m-1)  \leq  d-2.
      \]
\item If $p\notin\partial D$, then
      \[
         \sum_{i=1}^m \tau(D_i)-(d-m)  \leq  d-2.
      \]
\end{enumerate}
\end{corollary}

In particular, if all subdomains $D_i$ are full, then $\tau(D)=m-1$ when $p\in\partial D$, and $\tau(D)=0$ otherwise. 
Indeed, if $p\notin\partial D$, then $\deg_D(p)=d$.  Furthermore, each domain $D_i$ is glued at $p$ along a distinct edge, i.e. each $D_i$ contributes exactly one neighbor of $p$. Therefore 
\[
 m = \deg_D(p)=d.
\]
Since $\tau(D_i)=0$ for all $i$, Proposition~\ref{prop:tau-gluing} gives
\[
   \tau(D)=\sum_{i=1}^m \tau(D_i)-(d-m)=0.
\]
In both cases, $D$ is optimal (and full in the latter case). 
This also follows from Corollary~\ref{CorcritereIsop}, since $R(D)$ is either empty or equal to $\{p\}$, and hence has cardinality at most~$1$.

\medskip

\section{The combinatorial structure of  domains in $T_d$}
\label{sec.combstructure}

In this section, we describe a canonical decomposition of arbitrary domains in the regular tree $T_d$ (of cardinality at least $2$) 
into full components glued along the residual boundary.  We introduce the notion of \textit{stem}, a finite tree encoding this decomposition. This framework leads to an alternative criterion for isoperimetric optimality.

\subsection{Canonical decomposition}
\label{subsec.decomposition}

We first observe that full domains  are  indecomposable with respect to gluing  along boundary vertices:
\begin{lemma}\label{lem:full-no-boundary-gluing}
Let $D\subset T_d$ be a full domain.  Then $D$ cannot be written as a nontrivial gluing of domains along a boundary vertex.
\end{lemma}

\begin{proof}
We prove the contrapositive.  Assume that $D$ is obtained by gluing domains $D_1,\dots,D_m$ at some boundary vertex $p\in\partial D$, with $m\ge 2$.
Then
\[
\deg_D(p)=\sum_{i=1}^m \deg_{D_i}(p)\ge m\ge 2,
\]
so $p$ is not a leaf of $D$. Hence $\partial D\neq L(D)$, and therefore $D$ is not full.
\end{proof}

\medskip

Let $D\subset T_d$ be a non-full domain and denote by $B_1,\dots,B_m$ the connected components of $D\setminus R(D)$.  
Note that $m \geq 2$ else $D$ would be full, since in this case we would have $R(D) = \emptyset$.  For each $i$ we define
\begin{equation}\label{eq:def-Fi}
F_i  =  B_i \cup \bigl(R(D)\cap \partial' B_i\bigr)\subset D; 
\end{equation}
that is, $F_i$ consists of the component $B_i$ together with all vertices of $R(D)$ that are adjacent to $B_i$.
The number $m$ will be referred to as the \textit{number of full components of}  the domain $D$.  We will see in Remark \ref{rem.boundontaum} below that $m \leq  \tau(D) +1$. In particular, if $D$ is optimal, then  $m \leq d-1$.

\medskip  \newpage 

Our next theorem states that the $F_i$'s are full domains and $D$ is obtained by gluing these full domains along some vertices in $R(D)$:
\begin{theorem}\label{th.dec}
With the above notations, each $F_i$ is a full domain.  
Moreover, if we set
\[
   R_0  =  \{x\in R(D) \mid \deg_{R(D)}(x)=\deg_D(x)\},
\]
then we have
\begin{enumerate}[(i)]
 \item $R_0 = D \setminus \bigcup_{i=1}^m F_i$.
 \item $D =  R(D) \cup \left( \bigcup_{i=1}^m B_i \right) 
  =  R_0 \cup \left( \bigcup_{i=1}^m F_i \right)$.
\end{enumerate}
\end{theorem}

\medskip

Our proof will use the following Lemma about the neighbor sets of a vertex in $B_i$: 
\begin{lemma}\label{lem:neighbors-Ci}
For any $x\in B_i$ we have $N_{D}(x)=N_{F_i}(x)$. In particular $\deg_{F_i}(x)=\deg_D(x)$.
\end{lemma}

\begin{proof}
The inclusion $N_{F_i}(x)\subset N_D(x)$ is clear.
To prove the reverse inclusion, we show that every vertex $z\in N_D(x)$
belongs to $F_i$. We distinguish two cases.

\smallskip

Suppose first that $z\in R(D)$. Then $z\notin B_i$, but $z$ is adjacent to
$x\in B_i$. Hence $z\in \partial' B_i$, and therefore
\[
z\in R(D)\cap \partial' B_i \subset F_i,
\]
by definition of $F_i$.

\smallskip

Suppose now that $z\notin R(D)$. Then $z\in D\setminus R(D)$, and since $B_i$ is the connected component of $D\setminus R(D)$ containing $x$, it follows that $z\in B_i\subset F_i$. In particular, $z\in F_i$.

\smallskip

This proves that $N_D(x)\subset N_{F_i}(x)$, and hence $N_D(x)=N_{F_i}(x)$. 
\end{proof}

\medskip

\begin{proof}[Proof of Theorem \ref{th.dec}.]
Properties \textup{(i)}--\textup{(ii)} follow directly from the definitions, so it remains to prove that each $F_i$ is a full domain.
Since $B_i$ is connected and every vertex of $R(D)\cap \partial' B_i$ is adjacent to $B_i$, the set $F_i$ is connected. We therefore only need to show that every boundary vertex $x\in \partial F_i$ is a leaf of $F_i$. Two cases arise. 

\smallskip

If $x\in R(D)\cap \partial' B_i$, then by Lemma \ref{lem:outer-boundary-tree},   $x$ has exactly one neighbor in $B_i$ and no neighbor in $\partial' B_i$. Hence $x$ has exactly one neighbor in  $F_i = B_i \cup \bigl(R(D)\cap \partial' B_i\bigr)$, that is $x$ is a leaf of $F_i$.

\smallskip

If $x\in B_i$, then by Lemma~\ref{lem:neighbors-Ci}  we have  $\deg_{F_i}(x)=\deg_D(x)$. Since $x\in \partial F_i$, we have
$\deg_{F_i}(x)\le d-1$, hence also $\deg_D(x)\le d-1$ and therefore $x\in \partial D$, but since  $x\notin R(D)$, we conclude that $x\in L(D)$. Hence  $\deg_D(x)=1$ and thus also $\deg_{F_i}(x)=1$, i.e. $x$ is a leaf of $F_i$. 
\end{proof}

\medskip 

We will show in Section~\ref{sec.stemdiagrams} below how the boundary branching excess $\tau(D)$ can be expressed in terms of the combinatorial data associated with the above decomposition. This leads to a convenient criterion for isoperimetric optimality; see Lemma~\ref{lem.graphcount} and Theorem~\ref{th.reconstruction}.

\medskip

\subsection{The stem of a domain} \label{sec.stem}

In this subsection, we associate to any domain $D\subset T_d$ a quotient graph, called the \emph{stem}, which encodes the combinatorial pattern according to which the full components of $D$   are attached along the residual boundary $R(D)$. 
Let $B_1,\dots,B_m$ be the connected components of $D\setminus R(D)$, and recall that 
\[
  F_i = B_i \cup \bigl(R(D)\cap \partial' B_i\bigr)
\]
is a full domain for any $i = 1, \dots, m$ and $D$ is obtained by gluing these domains  along vertices in $R(D)$.

\medskip

\begin{definition}\label{def.stem}
The \emph{stem} of the domain $D$ is the graph $D^*$ obtained by contracting each connected component $B_i$ of $D\setminus R(D)$ to a single vertex, that we denote by  $b_i$, while leaving all vertices of $R(D)$ distinct. More precisely, the vertex set of $D^*$ is defined as the partitioned set
\begin{equation}\label{stempartition}
    V(D^*) = R(D)\,\sqcup\,\{b_1,\dots,b_m\},
\end{equation}
and the edges in $D^*$  are defined as follows:
\begin{enumerate}[\quad $\circ$\ ]
\item if $x,y\in R(D)$ are adjacent in $D$, then they are adjacent in $D^*$;
\item if $x\in R(D)$ and $y\in B_i$ are adjacent in $D$, then $x$ is adjacent to
      $b_i$ in $D^*$;
\item there are no edges between distinct vertices $b_i$ and $b_j$ if $i\neq j$.
\end{enumerate}
\end{definition}

\smallskip 

Note that $D^*$ is not a subgraph of $T_d$ but rather an abstract graph, in fact a quotient of $D$. Furthermore,  the partition \eqref{stempartition} is regarded as part of the structure of the stem. By construction, the stem records the adjacency relations between residual boundary vertices and the full components of $D$, while forgetting the internal structure of each full component. We now collect some basic properties of this graph.

\begin{proposition}\label{prop.stemproperties}
The stem $D^*$ of a domain $D\subset T_d$ satisfies the following properties:
\begin{enumerate}[(i)]
\item $D^*$ is a finite tree.
\item For every vertex $x\in R(D)$, we have 
\[
   2 \le \deg_{D^*}(x) = \deg_D(x) \le d-1.
\]
\item For each contracted vertex $b_i$,  we have 
\[
   \deg_{D^*}(b_i)  = |R(D)\cap \partial' B_i|
   = \bigl|\{x\in R(D)\mid x \text{ has a neighbor in } B_i\}\bigr|.
\]
\end{enumerate}
\end{proposition}

Observe that, by (ii), no leaf of $D^*$ belongs to $R(D)$. 

\medskip 

\begin{proof}
\emph{(i)} Since $D$ is a tree and each $B_i$ is connected, contracting the
subgraphs $B_i$ preserves connectedness and does not create cycles. Hence $D^*$
is a finite tree.

\smallskip

\emph{(ii)} Let $x\in R(D)$. Each edge incident to $x$ in $D$ either joins $x$ to
another vertex of $R(D)$ or to a vertex in some component $B_i$.
In the latter case, Lemma~\ref{lem:outer-boundary-tree} applied to $B_i$ implies
that $x$ has a unique neighbor in $B_i$.

Edges between vertices of $R(D)$ are unchanged by the contraction, while an edge
joining $x$ to $B_i$ becomes an edge joining $x$ to $B_i$ in $D^*$.
Therefore $\deg_{D^*}(x)=\deg_D(x)$, and the bounds follow from the definition of
the residual boundary.

\smallskip

\emph{(iii)} An edge incident to $b_i$ in $D^*$ corresponds precisely to an edge
of $D$ joining a vertex of $B_i$ to a vertex of $R(D)$.
Such vertices of $R(D)$ are exactly those in $R(D)\cap\partial' B_i$, which proves
the claim.
\end{proof}

\medskip 
 
\paragraph{Example}
Let us return  to  the example in Figure~\ref{figaa}. The domain is obtained by gluing three full domains $F_1,F_2,F_3$ at the common vertex $p$. Each $F_i$ is a full domain consisting of a vertex  of degree $5$ in $D$ together with its five neighbors. In each $F_i$ the vertex $p$ is a leaf.

\smallskip 

Likewise,  the example  in Figure~\ref{figab}  is obtained by gluing three full domains $F_1,F_2,F_3$ at the vertex $p$ and attaching three leaves at the vertex $q$, equivalently one glues two-vertex domains $F_4,F_5,F_6$ at $q$. 
In both cases, the stem is shown below.

\begin{figure}[H]
\centering
\begin{minipage}{0.48\textwidth}
\centering

\vspace{0.72cm}
\begin{tikzpicture}[scale=0.8,
  vertexRed/.style={circle,fill=red,inner sep=1.2pt},
  vertexBlue/.style={circle,fill=blue,inner sep=1.2pt},
  every path/.style={thin}
]
\node[vertexRed,label=below:$p$] (p) at (0,0) {};

\node[vertexBlue] (c1) at (2,0) {};
\node[vertexBlue] (c2) at (-1,1.7) {};
\node[vertexBlue] (c3) at (-1,-1.7) {};

\draw (p)--(c1) (p)--(c2) (p)--(c3);
\end{tikzpicture}
 \caption*{The stem of the domain in Figure \ref{figaa}.}
\end{minipage}
\hfill
\begin{minipage}{0.42\textwidth}
\centering
\begin{tikzpicture}[scale=0.8,
  vertexBlue/.style={circle,fill=blue,inner sep=1.2pt},
  vertexRed/.style={circle,fill=red,inner sep=1.2pt},
  every path/.style={thin}
]

\node[vertexRed,label=below:$p$] (p) at (0,0) {};
\node[vertexRed,label=right:$q$] (q) at (-1,1.7) {};

\node[vertexBlue] (c1) at (2,0) {};

\node[vertexBlue] (c3) at (-1,-1.7) {};

\node[vertexBlue] (a23) at (-1.8,0.9) {};

\draw (p)--(c1) (p)--(c2) (p)--(c3);
\draw[thin, red] (p)--(q);

\node[vertexBlue] (a23) at (-1.8,0.9) {};

\draw (p)--(a23);

\node[vertexBlue] (a21) at (-2,2.2) {};
\node[vertexBlue] (a22) at (-0.3,2.7) {};
\node[vertexBlue] (a24) at (0,1.1) {};
\draw (c2)--(a21) (c2)--(a22) (c2)--(a24);
\end{tikzpicture}
 \caption*{The stem of the domain in Figure \ref{figab}.}
\end{minipage}
\end{figure}

\subsection{Abstract stem diagrams} \label{sec.stemdiagrams} 

In order to investigate certain numerical relations arising in the stem of a domain, it is convenient to adopt an abstract viewpoint. 
To this end, we introduce the following notion.

\begin{definition}\label{def.stemdiagram}
A \emph{stem diagram} is an abstract graph $G^*$ whose vertex set is partitioned into two subsets
\[
V(G^*) = R \sqcup B.
\]
Vertices in $R$ are called \emph{red}, while those in $B$ are called \emph{blue}. We impose the following conditions:
\begin{enumerate}[(i)]
 \item $G^*$ is a finite tree;
 \item $B$ is an independent set (i.e., there are no blue--blue edges);
 \item all leaves of $G^*$ belong to $B$ (that is, every leaf is blue, although a blue vertex may have degree greater than~$1$).
\end{enumerate}
A stem diagram is said to be \emph{trivial} if $R = \emptyset$. In that case, $G^*$ consists of a single blue vertex.
\end{definition}
We write $r := |R|$ for the number of red vertices, $m := |B|$ for the number of blue vertices, and $l := |E(R)|$ for the number of red--red edges. We also define
\[
\tau^* = \tau^*(G^*) := \sum_{x \in R} \bigl(\deg_{G^*}(x) - 1\bigr).
\]

For a non-trivial stem diagram, the following basic inequalities hold:
\begin{equation}\label{basicstemequ}
  m\ge 2, \qquad 1\le r \le \tau^*, \qquad \text{and} \qquad r-l\ge 1.
\end{equation}
Indeed, since $G^*$ is a tree, it has at least two leaves; by definition all leaves
are blue, and therefore $m\ge 2$.
Moreover, non-triviality means that $R\neq \emptyset$, hence $r\ge 1$.
Since no red vertex is a leaf, we have $\deg_{G^*}(x)\ge 2$ for every $x\in R$, and thus
\[
  \tau^* =\sum_{x\in R}(\deg_{G^*}(x)-1)\ge |R|=r.
\]
Finally, the induced subgraph on $R$ is a forest, so $r-l=\chi(R)\ge 1$.

\medskip

\begin{lemma}\label{lem.graphcount}
With the above notation, we have
\[
 m+l=\tau^*+1.
\]
\end{lemma}

\begin{proof}
Let $n=|E(G^*)|$ denote the total number of edges of $G^*$. Write $l$ for the number of edges with both endpoints in $R$, and $k$ for the number of edges with one end in $R$ and the other end in $B$.
Since $B$ is an independent set, we have
\[
   n = l + k.
\]
Each red--red edge contributes $2$ to the sum of degrees
$\sum_{x\in R}\deg_{G^*}(x)$, while each red--blue edge contributes $1$.
Hence
\[
   \sum_{x\in R}\deg_{G^*}(x) = 2l + k = (l+k) + l  = n + l.
\]
It follows that
\[
   \tau^*  = \sum_{x\in R}\bigl(\deg_{G^*}(x)-1\bigr)
   = (n+l) - r.
\]
Since $G^*$ is a tree, we have $n=|G^*|-1=(r+m)-1$, and therefore
\[
 \tau^* = (r+m-1)+l-r = m+l-1.
\]
\end{proof}

\medskip

\begin{corollary}\label{cor.stem-bounds}
Let $G^*$ be a non-trivial stem diagram. With the notation introduced above, we have 
\[
   \tau^* +2 - r \le m \le \tau^* +1,
\]
and
\[
   \tau^* = |G^*| - 1 - \chi(R).
\]
In particular $\tau^* \le |G^*| - 2$, with equality if and only if the  subgraph $R \subset G^*$ is connected.
\end{corollary}

\begin{proof}
By Lemma~\ref{lem.graphcount}, we have $m=\tau^*+1-l$, which immediately gives $m\le \tau^* +1$.
Since the red induced subgraph is a forest, its Euler characteristic satisfies $\chi(R)=r-l\ge 1$, and therefore
\[
  m=\tau^*+1-l=\tau^*+1+\chi(R)-r\ge \tau^* +2-r.
\]

Moreover, using again $\chi(R)=r-l$, we compute
\[
  \tau^* = m+l-1 = m+(r-\chi(R))-1 = (r+m)-1-\chi(R)   = |G^*|-1-\chi(R).
\]
Since $\chi(R)\ge 1$, this implies $\tau^*\le |G^*|-2$, with equality  if and only if $\chi(R)=1$, that is, if and only if $R$ is connected.
\end{proof}

\medskip

\begin{remark} \label{rem.boundontaum}
Applying the previous corollary to an arbitrary domain $D \subset T_d$,  we obtain  
$$
 \tau(D)  = \tau^*(D^*)  \leq |D^*| -2,
$$
which improves the upper bound in~\eqref{inequtau1}.  
Furthermore, the number $m$ of full components of $D$ satisfies 
$$
 m \leq  \tau(D) +1.
$$ 
In particular, if $D$ is optimal, then $m \leq d-1$.

\end{remark}

\newpage 

\subsection*{Examples of stem diagrams}

We display below all isomorphism classes of stem diagrams with $\tau^*\le 3$ (note that each such stem is a finite tree with maximal degree at most $\tau^*+1$).

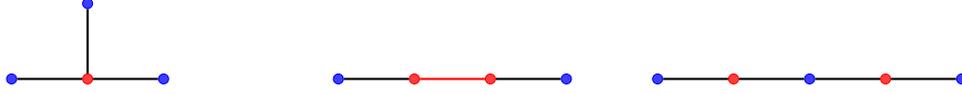
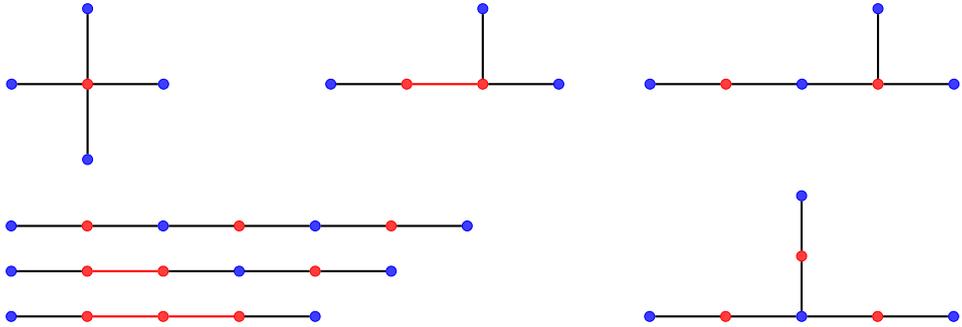
\begin{figure}[H]
\tikzset{
  redv/.style={circle, draw=red, fill=red!76, inner sep=1.32pt},
  bluev/.style={circle, draw=blue, fill=blue!76, inner sep=1.32pt},
  rededge/.style={thick, draw=red},
  blackedge/.style={thick, draw=black},
  tau/.style={font=\large}
}

\begin{tikzpicture}
\node[tau] at (-9,0) {$\tau^*=0$};
\node[bluev] (b0) at (-6.5,0) {};
\node[tau] at (-2.8,0) {$\tau^*=1$};
\node[bluev] (b1) at (-1,0) {};
\node[redv]  (r1) at (0,0) {};
\node[bluev] (b2) at (1,0) {};
\draw[blackedge] (b1)--(r1)--(b2);
\end{tikzpicture}
 
\vspace{1.2cm}

\begin{tikzpicture}
\node[tau] at (-6,1.4) {$\tau^*=2$};

\begin{scope}[xshift=-3.5cm]
\node[bluev] (a1) at (-1,0) {};
\node[redv]  (a2) at (0,0) {};
\node[bluev] (a3) at (1,0) {};
\node[bluev] (a4) at (0,1) {};
\draw[blackedge] (a1)--(a2)--(a3);
\draw[blackedge] (a2)--(a4);
\end{scope}

\hspace{0.3cm}

\begin{scope}[xshift=1cm]
\node[bluev] (b1) at (-1.5,0) {};
\node[redv]  (b2) at (-0.5,0) {};
\node[redv]  (b3) at (0.5,0) {};
\node[bluev] (b4) at (1.5,0) {};
\draw[blackedge] (b1)--(b2);
\draw[rededge]   (b2)--(b3);
\draw[blackedge] (b3)--(b4);
\end{scope}

\hspace{0.7cm}

\begin{scope}[xshift=5cm]
\node[bluev] (c1) at (-2,0) {};
\node[redv]  (c2) at (-1,0) {};
\node[bluev] (c3) at (0,0) {};
\node[redv]  (c4) at (1,0) {};
\node[bluev] (c5) at (2,0) {};
\draw[blackedge] (c1)--(c2)--(c3)--(c4)--(c5);
\end{scope}

\end{tikzpicture}

\vspace{1.5cm}

\begin{tikzpicture}
\node[tau] at (-6,1.4) {$\tau^*=3$};

\begin{scope}[xshift=-3.5cm]
\node[redv]  (d1) at (0,0) {};
\node[bluev] (d2) at (-1,0) {};
\node[bluev] (d3) at (1,0) {};
\node[bluev] (d4) at (0,1) {};
\node[bluev] (d5) at (0,-1) {};
\draw[blackedge] (d1)--(d2);
\draw[blackedge] (d1)--(d3);
\draw[blackedge] (d1)--(d4);
\draw[blackedge] (d1)--(d5);
\end{scope}

\hspace{0.7cm}

\begin{scope}[xshift=1cm]
\node[bluev] (e1) at (-2,0) {};
\node[redv]  (e2) at (-1,0) {};
\node[redv]  (e3) at (0,0) {};
\node[bluev] (e4) at (1,0) {};
\node[bluev] (e5) at (0,1) {};
\draw[blackedge] (e1)--(e2);
\draw[rededge]   (e2)--(e3);
\draw[blackedge] (e3)--(e4);
\draw[blackedge] (e3)--(e5);
\end{scope}

\hspace{0.7cm}

\begin{scope}[xshift=4.5cm]
\node[bluev] (f1) at (-2,0) {};
\node[redv]  (f2) at (-1,0) {};
\node[bluev] (f3) at (0,0) {};
\node[redv]  (f4) at (1,0) {};
\node[bluev] (f5) at (2,0) {};
\node[bluev] (f6) at (1,1) {};
\draw[blackedge] (f1)--(f2)--(f3)--(f4)--(f5);
\draw[blackedge] (f4)--(f6);
\end{scope}
\end{tikzpicture}

\vspace{0.3cm}

\hspace{2cm}
\begin{tikzpicture}

\begin{scope}[xshift=0cm,yshift=0cm]
\node[bluev] (h1) at (-3,0) {};
\node[redv]  (h2) at (-2,0) {};
\node[bluev]  (h3) at (-1,0) {};
\node[redv] (h4) at (0,0) {};
\node[bluev]  (h5) at (1,0) {};
\node[redv] (h6) at (2,0) {};
\node[bluev] (h7) at (3,0) {};
\draw[blackedge] (h1)--(h2)--(h3)--(h4)--(h5)--(h6)--(h7);
\end{scope}

\begin{scope}[xshift=0cm,yshift=-0.6cm]
\node[bluev] (h1) at (-3,0) {};
\node[redv]  (h2) at (-2,0) {};
\node[redv]  (h3) at (-1,0) {};
\node[bluev] (h4) at (0,0) {};
\node[redv]  (h5) at (1,0) {};
\node[bluev] (h6) at (2,0) {};
\draw[blackedge] (h1)--(h2);
\draw[rededge]   (h2)--(h3);
\draw[blackedge] (h3)--(h4)--(h5)--(h6);
\end{scope}

\begin{scope}[xshift=-0.5cm,yshift=-1.2cm]
\node[bluev] (g1) at (-2.5,0) {};
\node[redv]  (g2) at (-1.5,0) {};
\node[redv]  (g3) at (-0.5,0) {};
\node[redv]  (g4) at (0.5,0) {};
\node[bluev] (g5) at (1.5,0) {};
\draw[blackedge] (g1)--(g2);
\draw[rededge]   (g2)--(g3)--(g4);
\draw[blackedge] (g4)--(g5);
\end{scope}

\begin{scope}[xshift=7.4cm,yshift=-1.2cm]
\node[bluev] (k1) at (-2,0) {};
\node[redv]  (k2) at (-1,0) {};
\node[bluev] (k3) at (0,0) {};
\node[redv]  (k4) at (1,0) {};
\node[bluev] (k5) at (2,0) {};
\node[redv]  (k6) at (0,0.8) {};
\node[bluev] (k7) at (0,1.6) {};
\draw[blackedge] (k1)--(k2)--(k3)--(k4)--(k5);
\draw[blackedge] (k3)--(k6)--(k7);
\end{scope}
\end{tikzpicture}
\end{figure}
 
\bigskip  

The situation becomes significantly richer when $\tau^*\ge 4$. As an illustration, we describe a large family of stem diagrams with a prescribed value of \  $\tau^*\ge 1$.

\begin{enumerate}
\item
Start with an arbitrary finite tree $T$ having $\tau^*+2$ vertices.

\item
Turn $T$ into a stem diagram  $G^* = R\,  \sqcup B$ by coloring all leaves blue and all other vertices red.
Since the red  subgraph $R\subset G^*$ is connected, Corollary~\ref{cor.stem-bounds} implies
that the resulting stem diagram has branching excess exactly $\tau^*$.

\item
We  now generate further examples with the same value of $\tau^*$ by the following
local modification.
Choose a red--red edge $\{x,y\}$ and insert a new blue vertex $z$ in between.
The edge $\{x,y\}$ is then replaced by two red--blue edges $\{x,z\}$ and $\{y,z\}$.
This operation increases the number of blue vertices by one and decreases the number
of red--red edges by one.
By Lemma~\ref{lem.graphcount}, the value of $\tau^*$ remains unchanged.

\item
Repeating step 3 produces additional stem diagrams with the same prescribed value of~$\tau^*$.
\end{enumerate}

\smallskip 

This construction produces a large, albeit very special, family of stem diagrams with prescribed branching excess~$\tau^*$.
Already the choice of the initial tree on $n= \tau^*+2$ vertices 
grows exponentially with $\tau^*$ (see \cite{Otter1948} and \cite[Chapter~3]{Drmota2009} for precise asymptotics).
This shows that even within this restricted construction,
the combinatorial structure of stem diagrams rapidly becomes intricate for large~$\tau^*$.

\subsection{Reconstructing a domain from a stem diagram}
\label{sec.reconstruction}

In Sections~7.1--7.3, we proceeded in an \textit{analytic} way: we 
decomposed a domain $D \subset T_d$ into its stem $D^*$ and its full 
components $F_i$. We now proceed \textit{synthetically}, showing that 
this procedure can be inverted. More precisely, we prove that suitable 
abstract data---consisting of a stem diagram, a collection of full 
domains, and compatible gluing maps---determine a domain of $T_d$ 
uniquely up to isomorphism, yielding a complete combinatorial classification of all domains in $T_d$.
 
\medskip  

\begin{definition} \label{def.ARD}
Fix an integer $d \ge 2$. An admissible \emph{reconstruction datum} consists of a triple $(G^*, (F_i), (g_i))$, where 
\begin{enumerate}[(i)]
\item $G^* = R \sqcup \{b_1, \dots, b_m\}$ is an abstract stem diagram (Definition~\ref{def.stemdiagram}), that is, a finite tree whose vertex set is partitioned as
\[
V(G^*) = R \sqcup \{b_1, \dots, b_m\},
\]
such that the vertices $b_i$ are pairwise non-adjacent and, for every vertex $x \in R$, one has
\[
2 \le \deg_{G^*}(x) \le d-1,
\]

\item a collection of full domains $F_1, \dots, F_m$ in $T_d$ satisfying
\[
|\partial F_i| \ge \deg_{G^*}(b_i);
\]

\item and for each $i \in \{1, \dots, m\}$, an injective map
\[
g_i : N_{G^*}(b_i) \to \partial F_i,
\]
where $N_{G^*}(b_i) \subset R$ denotes the set of neighbors of $b_i$ in $G^*$.
\end{enumerate} 
\end{definition}

\medskip

\begin{remark}
As a concrete instance, the domain in Figure~2 (\S 4) corresponds to the admissible datum where $G^*$ is the stem shown at the end of \S 7. 2, the full 
components  $F_1, F_2, F_3$ are balls of radius $1$ in $T_5$, and the maps $g_i$ are the natural identifications along the unique 
leaf of each $F_i$ adjacent to $p$. The construction below then recovers $D$ up to isomorphism.
\end{remark}

\medskip

\paragraph{\textbf{Construction of a graph $G$.}}
Given admissible gluing data $(G^*,(F_i),(g_i))$, we construct a graph $G$ as follows. Consider first the disjoint union
\[
X = R \sqcup \bigsqcup_{i=1}^m F_i,
\]
and for each $i$ and each $x\in N_{G^*}(b_i)$, identify the vertex $x\in R$ with the vertex $g_i(x)\in \partial F_i$.
This defines a quotient graph that we denote by $G = X/\sim$. We denote by $\omega : X \rightarrow G$ the canonical projection. 
Now, for each $i = \{1, \dots, m\}$, we set
\[
  F_i' := F_i \setminus g_i\!\bigl(N_{G^*}(b_i)\bigr) \subset X  \quad\text{and}\quad  B_i := \omega(F_i')\subset G,
\]
and we define a map 
\begin{equation}\label{def.mappi}
  \pi : G \rightarrow G^*,  \qquad  \pi(x) = \begin{cases} x & \text{ if }  x\in R, \\ 
   b_i & \text{ if }  x\in B_i. \end{cases}
\end{equation}

We collect some basic properties of this construction in the following 
\begin{lemma}\label{lem:G-is-tree}
The graph $G$ is a finite tree and the degree of each vertex in $G$ is at most $d$.
Moreover:
\begin{enumerate}[\hspace{0.3cm} (a)]
\item The map $\omega$ is injective on $R$ and on each $F_i'$.
\item The vertex set of $G$ decomposes as the disjoint union $V(G)=\omega(R)\;\sqcup\; B_1 \;\sqcup\; \cdots \;\sqcup\; B_m.$
\item The map $\pi$ restricts to the identity on $\omega(R)$ and collapses each $B_i$ to the vertex $b_i$.
\item The sets $B_1,\dots,B_m$ are precisely the connected components of $G\setminus \omega(R)$.
\end{enumerate}
\end{lemma}

\begin{proof}
Points (a)--(c) follow directly from the construction of $G$.

\smallskip

To prove (d),  observe first that each $F_i'$ is connected. Indeed, since each
$F_i$ is full, every vertex of $\partial F_i$ is a leaf of $F_i$, and by
Lemma~\ref{lem:leaf-connected-complement}, removing finitely many leaves  from a tree does not disconnect it. Hence each $F_i'$ is connected, and therefore
its image $B_i=\omega(F_i')$ is connected. 
It now follows from (b) that  $B_1,\dots,B_m$ are precisely the connected components of $G\setminus \omega(R)$.
 
\medskip

We now prove that $G$ is a finite tree.  Finiteness follows from (b), since $R$ and each $F_i'$ are finite.

\medskip

\emph{$G$ is connected.}
By (a), $\omega$ restricts to a graph isomorphism between $R$ and $\omega(R)$.
 Since $R$ is connected (as a subtree of the tree $G^*$), it follows that $\omega(R)$ is connected.
Since each $B_i$ meets $\omega(R)$  along at least one vertex, $G$ is also connected.

\medskip

\emph{$G$ is acyclic.} Indeed, if  $G$ contained a non-trivial cycle, its image under $\pi$ would give a non-trivial closed walk in $G^*$. Since $\pi$ collapses only the connected subgraphs $B_i$,
this would produce a cycle in $G^*$, contradicting that $G^*$ is a tree.

\medskip

We finally prove that the  degree of any vertex $x$ in $G$ is at most $d$. By (b), we know that $x$ either belongs to $\omega(R)$ or to some $B_i$. If $x\in B_i$, there exists a vertex $z\in F_i'$ such that $\omega(z) =x$, thus
$$
 \deg_{G}(x)  =  \deg_{F_i}(z) \leq d  
$$
since $F_i$ is by definition a domain in $T_d$. 
On the other hand, if  $x\in \omega(R)$, each neighbor $b_i$ in $G^*$ corresponds to exactly one edge in $G$, since $x$ is identified with a unique boundary leaf of $F_i$, which has a unique neighbor inside $F_i$.
Thus
\[
\deg_G(x)=\deg_{G^*}(\pi(x))\le d-1.
\]
Hence $\max_{x\in G}\deg_G(x)\le d$.

\end{proof}

 \medskip

We  now state the reconstruction Theorem: 

\medskip

\begin{theorem} \label{th.reconstruction}
There exists a graph isomorphism $\phi : G \to D$ from $G$ onto a domain $D\subset T_d$.
It satisfies the following properties 
\begin{enumerate}[(i)]
\item The $\phi$-image of $\omega(R)$ coincides with the residual boundary of $D$.
\item the full components of $D$ are the images of the graphs $F_i$;
\item The map $\phi$ induces an isomorphism $\phi^*$ between $G^*$ and the stem $D^*$ of $D$.
\end{enumerate}
Moreover,
\[
\tau(D)=\tau^*(G^*)=m+l-1,
\]
where $l=|E(R)|$ is the number of edges in the subgraph $R\subset G^*$. In particular, $D$ is optimal if and only if
\[
 m+l\le d-1.
\]
Conversely, every domain $D\subset T_d$ arises, up to isomorphism, from admissible gluing data $(G^*,(F_i),(g_i))$ by this construction.
\end{theorem}

\begin{proof}
The existence of  $\phi$ follows from Lemmas \ref{lem:embedding}  and \ref{lem:G-is-tree}. 
For simplicity, we henceforth regard $R$ as a subset of both $G$ and $D$ via 
the maps $\omega$ and $\phi$. This identification is justified since 
$\omega$ is injective on $R$ by Lemma~\ref{lem:G-is-tree}(a)
, and $\phi$ is a graph 
isomorphism; in particular, both maps preserve the combinatorial 
structure of $R$.
We prove (i)--(iii) : 

\smallskip

(i) 
Let $x\in R$. By Definition~\ref{def.ARD}(i),
\[
2 \le \deg_{G^*}(x) \le d-1.
\]
Since $\pi$ restricts to the identity on $R$, we have
\[
\deg_G(x)=\deg_{G^*}(x).
\]
After embedding $G$ into $T_d$, the vertex $x$ therefore has degree between $2$ and $d-1$ inside $D$.
Because $T_d$ is $d$–regular, $x$ has at least one neighbor in $T_d\setminus D$, and since its degree in $D$ is at least $2$, it is not a leaf. Thus $x$ belongs to the residual boundary of $D$.

\smallskip

Conversely, if $y\in B_i$, then $y$ lies in the image of $F_i'$. Its degree inside $D$ coincides with its degree in the full domain $F_i$. Since interior vertices of a full domain are not in the residual boundary, we conclude that the residual boundary of $D$ is precisely $R$.

\smallskip

(ii)  Since $F_i = B_i \cup (R \cap \partial' B_i)$ inside $G$, its image under $\phi$ is
$$
 \phi(F_i)=\phi(B_i)\cup\bigl(\phi(R)\cap\partial'\phi(B_i)\bigr).
$$
These  are the full components of $D$ by definition \eqref{eq:def-Fi}.
 
\smallskip

(iii)  
Collapsing each full component of $D$ corresponds exactly to the map $\pi$,
so the stem  $D^*$ of $D$ is isomorphic to $G^*$.

\smallskip

Using Lemma~\ref{lem.graphcount} and $D^* \simeq G^*$, we have now
$$
\tau(D)=\tau^*(D^*) = \tau^*(G^*)=m+l-1.
$$
The optimality criterion follows then from  Theorem \ref{CritereIsop} since the condition 
$\tau(D) \leq d-2$ is equivalent to $m+l =  \tau(D) + 1\le d-1$.

\smallskip

In the converse direction, given  a domain $D\subset T_d$, its stem $D^*$,
its full components, and the natural identification maps along the residual boundary
provide admissible  gluing data.
Applying the above construction produces  a graph isomorphic to $D$.
\end{proof}

\medskip

The situation is summarized in the following commutative diagram, 
where $\rho : D \to D^*$ is the canonical projection: 
\[
\begin{tikzcd}[row sep=2.2em, column sep=2.5em]
X \arrow[r, "\omega"] \arrow[dr, "\pi\circ\omega"'] &
G \arrow[d, "\pi"] \arrow[r, "\phi"] &
D \arrow[d, "\rho"] \\
& G^* \arrow[r, "\phi^*"] & D^*
\end{tikzcd}
\]
The maps $\phi$ and $\phi^*$ are graph isomorphisms while $\omega$,  $\pi$ and $\rho$ are  epimorphisms.

\medskip 

The decomposition and reconstruction procedures described in this section are inverse to each other up to isomorphism. Hence domains in $T_d$ are in natural one--to-one correspondence with  admissible gluing data.

\bigskip 

\textit{Acknowledgements.} The author thanks Bruno Luiz Santos Correia for discussions on this subject during the preparation of his PhD thesis and for comments on an earlier version of the manuscript. He also thanks Oliver Janzer for helpful comments and suggestions, and Nicolas Monod for pointing out the link with the notion of complete subtrees in representation theory and for suggesting the reference~\cite{Olshanskii1977}.

\end{document}